\documentclass[10pt,reqno]{article}

\pdfoutput = 1

\usepackage[english]{babel}
\usepackage{amsmath,amssymb,amsthm}
\usepackage{amsfonts}
\usepackage{amsxtra}

\usepackage{array,dcolumn}
\usepackage{graphicx}
\usepackage[hyperfootnotes=true,
        pdffitwindow=true,
        plainpages=false,
        pdfpagelabels=true,
        pdfpagemode=UseOutlines,
        pdfpagelayout=SinglePage,
        hyperindex,]{hyperref}

\usepackage[T1]{fontenc}
\usepackage[utf8]{inputenc}
\usepackage[activate={true,nocompatibility},spacing,kerning]{microtype}
\usepackage[sc,osf]{mathpazo}
\microtypecontext{spacing=nonfrench}

\newcommand{\mathsym}[1]{{}}
\newcommand{\unicode}[1]{{}}

\theoremstyle{plain}
\newtheorem{theorem}{Theorem}

\newtheorem{corollary}[theorem]{Corollary}
\newtheorem{proposition}[theorem]{Proposition}

\theoremstyle{definition}

\theoremstyle{remark}
\newtheorem{remark}[theorem]{Remark}

\newcommand{\Z}{\mathbb Z}
\newcommand{\R}{\mathbb R}

\renewcommand{\leq}{\leqslant}
\renewcommand{\geq}{\geqslant}

\begin{document}

\begin{center}
{\bfseries\Large Parametrising correlation matrices}
\\[2\baselineskip]
P. J. Forrester\footnote{pjforr@unimelb.edu.au} and Jiyuan Zhang\footnote{jiyuanz@student.unimelb.edu.au}%

{\itshape ARC Centre of Excellence for Mathematical and Statistical Frontiers,\\
School of Mathematics and Statistics, The University of Melbourne, Victoria 3010, Australia.}\\

\end{center}

\begin{abstract}
\noindent
Correlation matrices are the sub-class of positive definite real matrices with all entries on the diagonal equal to unity. Earlier work has exhibited
a parametrisation of the corresponding Cholesky factorisation in terms of partial correlations, and also in terms of hyperspherical co-ordinates.
We show how the two are relating, starting from the definition of the partial correlations in terms of the Schur complement. We extend this to the
generalisation of correlation matrices to the cases of complex and quaternion entries. As in the real case, we show how the hyperspherical
parametrisation leads naturally to a distribution on the space of correlation matrices $\{R\}$ with probability density function proportional to
$( \det R)^a$. For certain $a$, a construction of random correlation matrices realising this distribution is given in terms of rectangular standard
Gaussian matrices.
\end{abstract}

\section{Introduction}\label{s1}

In applications of matrices, there are many settings in which the rows and columns have distinct meaning. For example, in a survey of $n$ people, giving a numerical score between $0$ and $5$ for their rating of $N$ different movies, there is an $n\times N$ matrix $X$---the data matrix---such that the rows correspond to the people and the columns to the movies. The $k$-th column $\mathbf X^{(k)}$ is then the vector of scores given for movie $k$, and its $j$-th entry is the score given by person $j$. Let $\mu_k$ denote the average of the scores in column $k$, and let $\mathbf{1}_n$ denote the $n\times 1$ vector with all entries equal to 1. The recentred, zero mean score vectors are then specified as $\mathbf Y^{(k)}=\mathbf X^{(k)}-\mu_k\mathbf 1_n$, and the recentred data matrix is
\begin{equation}\label{0}
Y=\begin{bmatrix}\mathbf Y^{(1)}&\mathbf Y^{(2)}&\cdots&\mathbf Y^{(N)}\end{bmatrix}
\end{equation}

The sample covariance matrix $S$ is specified in terms of $Y$ as
\begin{equation}\label{1}
S=\frac{1}{n-1}Y^\top Y.
\end{equation}
Note that $S$ is a $N\times N$ symmetric matrix, and its entry in rows $j$ and column $k$ gives the sample covariance between the scores of movies $j$ and $k$. In the case that the rows and/or columns of $Y$ are drawn from a vector Gaussian distribution with given  covariance, \eqref{1} is referred to a Wishart matrix. For such random matrices, a vast number of theoretical results have been assembled, and applied settings identified, since the pioneering paper \cite{Wi28};
see e.g.~\cite{Mu82}.

Natural from the viewpoint of data analysis is to further refine \eqref{1} by forming the sample correlation matrix
\begin{equation}\label{2}
R=\begin{bmatrix}
\displaystyle\frac{\left(\mathbf Y^{(j)}\right)^\top \mathbf Y^{(k)}}{\left\|\mathbf Y^{(j)}\right\|\,\left\| \mathbf Y^{(k)}\right\|}
\end{bmatrix}_{j,k=1}^N=:\begin{bmatrix}\rho_{jk}
\end{bmatrix}_{j,k=1}^N.
\end{equation}
Here, as well as the original score vectors being centred by subtracting their mean, each has been scaled to correspond to a unit vector. One sees immediately that the entries of $R$ are all equal to unity on the diagonal, while on the off diagonal, in accordance with the Cauchy-Schwarz inequality, they all have modulus less than or equal to 1. Moreover, the decomposition
\begin{equation}\label{3}
R=\mathbf D^\top Y^\top Y \mathbf D,
\end{equation}
where
\begin{equation*}
\mathbf D=\begin{bmatrix}
\displaystyle\frac{1}{\left\| \mathbf Y^{(j)}\right\|}
\end{bmatrix}_{j=1}^N
\end{equation*}
shows that $R$, like $S$, is positive definite but now with bounded entries $|\rho_{jk}|\leq 1$.

This latter feature, although making some aspects of theoretical analysis more difficult (e.g.~studies of the eigenvalues) allows for a distinct set of questions to be posed. For example, in the case of correlation matrices the volume of the natural embedding in $\R^{N(N-1)}$---referred to as an elliptope \cite{LP95} and to be denoted $\mathcal R_N$---is well defined. Knowledge of this volume allows an answer to the question: if the strictly upper triangular entries of \eqref{2} are chosen uniformly at random in the range $(-1,1)$, what is the probability that $R$ is a valid correlation matrix (i.e.~is positive definite) \cite{EHNS16}?

A direct approach to this question requires a parametrisation of the space of correlation matrices. Two such parametrisations are available in the literature, both applying to the lower triangular matrix $L$ in the Cholesky factorisation
\begin{equation}\label{4}
R=LL^\top.
\end{equation}
One of these use hyperspherical co-ordinates in $\R^j$ to parametrise row $j$ ($j=1,\cdots ,N$) \cite{PB96,RJ00,RBM07,PW15,EHNS16}, and the other makes use of a sequence of partial correlations \cite{Jo06,LKJ09,Hu12}. The latter method yielded the first direct computation of the volume \cite{Jo06}
\begin{equation}\label{5}
\mathrm{vol}\left(\mathcal R_N\right)=\prod_{j=2}^N 2^{(j-1)^2}\left(B(j/2,j/2)\right)^{j-1}
\end{equation}
where
\begin{equation}\label{5.1}
B(a,b)=\frac{\Gamma(a)\Gamma(b)}{\Gamma(a+b)}.
\end{equation}
It is only in the last few years that this same formula (in equivalent forms) was derived using the hyperspherical parametrisation \cite{PW15,EHNS16}.

Indirect computations of $\mathrm{vol}\left(\mathcal R_N\right)$ are also possible. Such a method, giving a formula equivalent to \eqref{5} actually predates the work \cite{Jo06}---this is due to Wong et al. \cite{WCK03}. As implied by a comment in \cite[2nd paragraph of Introduction]{PW15} the same result, again deduced indirectly, follows from the still earlier work of Muirhead \cite[p.148]{Mu82}. 

The circumstances just described suggest a number of follow up problems. The most immediate is to relate the hyperspherical and partial correlation parametrisations. To give a satisfactory account on this point, a self contained theory relating to the latter must be developed. Moreover, the hyperspherical parametrisation gives a different viewpoint on known results \cite{Mu82} for the marginal distribution of the elements of \eqref{2}, when chosen uniformly at random, and similarly for the moments of $\det R$.

The literature cited above is restricted to the case of real entries. Complex valued covariance matrices, and thus complex valued correlation matrices, are well motivated from the viewpoint of their application in wireless communication; see for example \cite{TV04}. Thus, in addition to addressing the above problems when the correlation matrices have real entries, we consider too the case of complex (and quaternion) entries.

\setcounter{equation}{0}
\section{Cholesky factorisation and parametrisations}\label{S2}
Let $R$ be an $N\times N$ positive definite matrix with all diagonal entries equal to unity, as is consistent with \eqref{2}. Introduce the Cholesky factorisation \eqref{4} with
\begin{equation}\label{2.1}
L=
\begin{bmatrix}
l_{11}&0&0&\cdots&0\\
l_{21}&l_{22}&0&\cdots&0\\
l_{31}&l_{32}&l_{33}&\cdots&0\\
\vdots&\vdots&\vdots&\ddots&\vdots&\\
l_{N1}&l_{N2}&l_{N3}&\cdots& l_{NN}
\end{bmatrix}.
\end{equation}
Well established theory (see e.g. \cite{GvL83}) gives that this is unique for $R$ positive definite subject to the requirement that
\begin{equation}\label{2.2}
l_{jj}>0,\ \ (j=1,\cdots, N).
\end{equation}

The fact that the diagonal entries in \eqref{2} are all equal to unity implies that the sum of the squares of the non- zero entries along each rows $j$ of $L$ is also unity,
\begin{equation}\label{2.2a}
\sum_{k=1}^j l_{jk}^2=1.
\end{equation}
Hence $(l_{j1},l_{j2},\cdots,l_{jj})$ is a point on the sphere $S_{j-1}$. As such it permits the hyperspherical parametrisation (see references noted below \eqref{4})
\begin{equation}\label{2.3}
l_{jk}=\begin{cases}
\cos\theta_{jk}\displaystyle\prod\limits_{p=1}^{k-1}\sin\theta_{jp}&(1\leq k\leq j-1)\\
\displaystyle\prod\limits_{p=1}^{j-1}\sin\theta_{jp}&(j=k),\\
\end{cases}
\end{equation}
where for $k=1$ the products are to be taken as equal to unity. The requirement \eqref{2.2} implies
\begin{equation}\label{2.4}
0<\theta_{j,k}<\pi\ \ \ \ (1\leq k < j \le N).
\end{equation}

Let us now turn our attention to the parametrisation of the entries in \eqref{2.1} using partial correlation coefficients. In the setting leading to the definition \eqref{2} one defines the partial correlation coefficients $\rho_{j,k|\{1,\cdots,p-1\}}$ ($p\le j, k\le N$) as the entries of the $(N-p+1)\times(N-p+1)$ matrix
\begin{equation*}
\begin{bmatrix}
\rho_{j,k|\{1,\cdots,p-1\}}
\end{bmatrix}_{j,k=p}^N=
\begin{bmatrix}
\displaystyle\frac{\left(\mathbf Y^{(j)}-P_{\{1,\cdots,p-1\}}^\perp\mathbf Y^{(j)}\right)\cdot\left(\mathbf Y^{(k)}-P_{\{1,\cdots,p-1\}}^\perp\mathbf Y^{(k)}\right)}{\left\|\mathbf Y^{(j)}-P_{\{1,\cdots,p-1\}}^\perp\mathbf Y^{(j)}\right\|\,\left\|\mathbf Y^{(k)}-P_{\{1,\cdots,p-1\}}^\perp\mathbf Y^{(k)}\right\|}
\end{bmatrix}_{j,k=p}^N
\end{equation*}
where $P_{\{1,\cdots,p-1\}}^\perp\mathbf x$ denotes the orthogonal projection of the  vector $\mathbf x$ onto the hyperplane spanned by $\mathbf Y^{(1)},\cdots,\mathbf Y^{(p-1)}$ \cite{An58}.

The partial correlation coefficients are intimately related to the Schur complement of $R$ \cite{Ou81}. In relation to the latter, partition $R$ according to
\begin{equation*}
R=\begin{bmatrix}
R_{11}&R_{12}\\R_{21}&R_{22}
\end{bmatrix},
\end{equation*}
where $R_{21}=R_{12}^\top$ is of size $(N-p+1)\times(p-1)$, $R_{11}$ is of size $(p-1)\times(p-1)$ and $R_{22}$ is of size $(N-p+1)\times(N-p+1)$. The Schur complement is the $(N-p+1)\times(N-p+1)$ matrix
\begin{equation}\label{2.5}
R/R_{11} :=\begin{bmatrix}
\left(R/R_{11}\right)_{j,k}
\end{bmatrix}_{j,k=p}^N :=R_{22}-R_{21}R_{11}^{-1}R_{12}.
\end{equation}
In terms of the Schur complement, the partial correlations can be written \cite{An58,Ou81}
\begin{equation}\label{2.6}
\rho_{jk|\{1,\cdots,p-1\}}=\frac{(R/R_{11})_{jk}}{(R/R_{11})_{jj}^{1/2}(R/R_{11})_{kk}^{1/2}}
\end{equation}

It is stated in \cite[$\S$3.3]{CJA10} that the elements in (\ref{2.1}) as required for the Cholesky factorisation \eqref{4} can be writen in terms of a subset of the partial correlations according to
\begin{align}\label{2.7}
l_{j1}&=\rho_{j1}& (j=1,\cdots,N)\nonumber\\
l_{jk}&=\rho_{jk|\{1,\cdots,k-1\}}\prod_{p=1}^{k-1}\sqrt{1-\rho_{jp|\{1,\cdots,p-1\}}^2}& (j=k+1,\cdots,N)\nonumber\\
l_{jj}&=1-\sum_{p=1}^{j-1}l_{jp}^2& (j=2,\cdots,N).
\end{align}
Comparison with the hyperspherical parametrisation \eqref{2.3} shows the simple relationship
\begin{equation}\label{2.8}
\rho_{jk|\{1,\cdots,k-1\}}=\cos\theta_{jk}, \quad 1\leq k< j\leq N.
\end{equation}

Following ideas \cite{DIP88} we will show how \eqref{2.8} can be derived in a way that is consistent with \eqref{2.6}.
With $Y$ given by \eqref{0}, this requires introducing the covariance matrix
\begin{equation}\label{2.9}
Y^\top Y=\Sigma=\begin{bmatrix}\sigma_{jk}
\end{bmatrix}=\begin{bmatrix}
\Sigma_{\mathbf p}&\Sigma_{\mathbf {pq}}\\
\Sigma_{\mathbf {qp}}&\Sigma_{\mathbf q}
\end{bmatrix}.
\end{equation}
Here $\mathbf p=\{1,\cdots,p-1\}$, $\mathbf q=\{p,\cdots,N\}$ and $\Sigma_{\mathbf{pq}}$ is the sub-block of $\Sigma$ formed from rows $\mathbf p$ and columns $\mathbf q$ ($\Sigma_{\mathbf p}$ is an abbreviation for $\Sigma_{\mathbf p\mathbf p}$). As is consistent with \eqref{2.5} the Schur complement of $\Sigma$ corresponding to the partition \eqref{2.9} is
\begin{equation}\label{2.10}
\Sigma/\Sigma_\mathbf{p}=\begin{bmatrix}
\sigma_{jk|\{1,\cdots,p-1\}}
\end{bmatrix}_{j,k=p}^N=\Sigma_{\mathbf q}-\Sigma_{\mathbf {qp}}\Sigma_{\mathbf p}^{-1}\Sigma_{\mathbf {pq}}.
\end{equation}
The quantities $\sigma_{jk|\{1,\cdots,p-1\}}$ are referred to as partial covariances.

As noted in \cite{DIP88}, \eqref{2.10} can be written in terms of a suitable block partitioning of the Cholesky factorisation of $\Sigma$,
\begin{equation}\label{2.11}
\Sigma=\begin{bmatrix}\sigma_{jk}\end{bmatrix}_{j,k=1}^N=AA^\top=\begin{bmatrix}
A_{\mathbf p}&0_\mathbf{pq}\\
A_{\mathbf {qp}}&A_{\mathbf q}
\end{bmatrix}\begin{bmatrix}
A_{\mathbf p}^\top&A_{\mathbf {qp}}^\top\\
0^\top_\mathbf{pq}&A_{\mathbf q}^\top
\end{bmatrix},
\end{equation}
where $A$ is a lower triangular matrix with entries on the diagonal nonnegative and $A_\mathbf p, A_{\mathbf{qp}}$ etc are sub-blocks as in the notation \eqref{2.9}. Thus one sees that
\begin{equation}\label{2.12}
\Sigma/\Sigma_\mathbf{p}=A_{\mathbf {qp}}A_{\mathbf {qp}}^\top+A_{\mathbf {q}}A_{\mathbf {q}}^\top-A_{\mathbf {q p}}A_{\mathbf {p}}^\top(A_{\mathbf {p}}A_{\mathbf {p}}^\top)^{-1}A_{\mathbf {p}}A_{\mathbf {q p}}^\top=A_{\mathbf {q}}A_{\mathbf {q}}^\top.
\end{equation}
Using this, a known recursive formula for the partial covariances can be re-derived in a structured, self contained way, which moreover relates to \eqref{2.8}.
\begin{proposition}
The partial covariances satisfy the recursive formula \cite{An58}
\begin{equation}\label{2.13}
\sigma_{jk|\{1,\cdots,p\}}=\sigma_{jk|\{1,\cdots,p-1\}}-\frac{\sigma_{jp|\{1,\cdots,p-1\}}\sigma_{pk|\{1,\cdots,p-1\}}}{\sigma_{pp|\{1,\cdots,p-1\}}}, 
\quad
p+1\leq j,k\leq N.
\end{equation}	
\end{proposition}
\begin{proof}
Extend \eqref{2.9} by writing
\begin{equation}
\Sigma=
\begin{bmatrix}
\Sigma_{\mathbf p}&\begin{bmatrix}\sigma_{jp}\end{bmatrix}_{j=1}^{p-1}&\Sigma_{\mathbf {p\tilde{\mathbf q}}}\\
\begin{bmatrix}\sigma_{pk}\end{bmatrix}_{k=1}^{p-1}&\sigma_{pp}&\begin{bmatrix}\sigma_{pk}\end{bmatrix}_{k=p+1}^{N}\\
\Sigma_{\mathbf {\tilde{\mathbf q}p}}&\begin{bmatrix}\sigma_{jp}\end{bmatrix}_{j=p+1}^{N}&\Sigma_{\tilde{\mathbf q}}
\end{bmatrix},
\end{equation}
where $\tilde{\mathbf q}=\{p+1,\cdots,N\}$. Similarly, in the Cholesky factorisation (\ref{2.11}), extend the block decomposition of the lower triangular matrix $A$ by writing
\begin{equation}\label{2.14}
A=
\begin{bmatrix}
A_{\mathbf p}&\begin{bmatrix}0\end{bmatrix}_{j=1}^{p-1}&0_{\mathbf {p\tilde{\mathbf q}}}\\
\begin{bmatrix}A_{pk}\end{bmatrix}_{k=1}^{p-1}&A_{pp}&\begin{bmatrix}0\end{bmatrix}_{k=p+1}^{N}\\
A_{\mathbf {\tilde{\mathbf q},p}}&A_{\mathbf {\tilde{\mathbf q}},\{p\}}&A_{\tilde{\mathbf q}}
\end{bmatrix}.
\end{equation}
Substituting the bottom $2\times 2$ block of \eqref{2.14} for $A_\mathbf q$ in \eqref{2.12} shows
\begin{equation}\label{2.15}
\Sigma/\Sigma_{\{1,\cdots,p-1\}}=\begin{bmatrix}
A_{pp}^2&A_{pp}A_{\tilde{\mathbf {q}},\{p\}}^\top\\
A_{pp}A_{\tilde{\mathbf {q}}, \{p\} }&A_{\tilde{\mathbf {q}},\{p\}}A_{\tilde{\mathbf {q}}, \{p\} }^\top+A_{\tilde{\mathbf {q}}}A_{\tilde{\mathbf {q}}}^\top\\
\end{bmatrix}.
\end{equation}
In relation to $A_{\tilde{\mathbf {q}}}A_{\tilde{\mathbf {q}}}^\top$ as appearing in this expression, we note from \eqref{2.12} with $\mathbf p$ replaced by $\{1,\cdots,p\}$ and $\mathbf q$ by $\tilde{\mathbf q}$, that
\begin{equation}\label{2.16}
\Sigma/\Sigma_{\{1,\cdots,p\}}=A_{\tilde{\mathbf {q}}}A_{\tilde{\mathbf {q}}}^\top.
\end{equation}
Recalling the notation in\eqref{2.10} for the partial covariances, we read off from \eqref{2.15} that
\begin{align}
\sigma_{jk|\{1,\cdots,p-1\}}&=A_{jp}A_{pk}+\left(A_{\tilde{\mathbf{q}}}A_{\tilde{\mathbf{q}}}^\top\right)_{jk}&(p+1\leq j,k\leq N)\label{2.17a}\\
\sigma_{jp|\{1,\cdots,p-1\}}&=A_{pp}\left(A_{\tilde{\mathbf{q}}, \{p\}}\right)_{jp}=A_{pp}A_{jp}&\label{2.17b}\\
\sigma_{pk|\{1,\cdots,p-1\}}&=A_{pp}\left(A_{\tilde{\mathbf{q}}, \{p\} }\right)_{pk}=A_{pp}A_{pk}&\label{2.17c}\\
\sigma_{pk|\{1,\cdots,p-1\}}&=A_{pp}^2.&\label{2.17d}
\end{align}
Similarly, we read off from \eqref{2.16} that
\begin{equation}\label{2.18}
\sigma_{jk|\{1,\cdots,p\}}=\left(A_{\tilde{\mathbf{q}}}A_{\tilde{\mathbf{q}}}^\top\right)_{jk}\ \ \ \ (p+1\leq j,k\leq N).
\end{equation}
Eliminating all dependence on the matrix $A$ and its entries in \eqref{2.17a}–\eqref{2.17d}, \eqref{2.18} gives \eqref{2.13}
\end{proof}

The recurrence \eqref{2.13}, obtained through different working described as 'tedious', is given in \cite[\S 2.5.3]{An58}. And as noted in this latter reference, it follows immediately from \eqref{2.13}, and the relation between partial correlations and partial covariances
\begin{equation}\label{2.19}
\rho_{jk|\{1,\cdots,p-1\}}:=\displaystyle\frac{\sigma_{jk|\{1,\cdots,p-1\}}}{ \displaystyle\sqrt{\sigma_{jj|\{1,\cdots,p-1\}}\sigma_{kk|\{1,\cdots,p-1\}}}},\ \ \ \ p\leq j,k\leq N
\end{equation}
(compare \eqref{2.6} and \eqref{2.10}), that there is a similar recurrence to \eqref{2.13} for the partial correlations.
\begin{corollary}
We have
\begin{equation*}
\rho_{jk|\{1,\cdots,p\}}=\frac{\rho_{jk|\{1,\cdots,p-1\}}-\rho_{jp|\{1,\cdots,p-1\}}\rho_{pk|\{1,\cdots,p-1\}}}{\sqrt{1-\rho_{jp|\{1,\cdots,p-1\}}^2}\sqrt{1-\rho_{pk|\{1,\cdots,p-1\}}^2}},\ \ \ \ p+1\leq j,k\leq N.
\end{equation*}
\end{corollary}

Most significant in relation to explaining the parametrisation \eqref{2.8} are the relations \eqref{2.17b} and \eqref{2.17c}. Thus substituting in \eqref{2.12}, recalling the first equality in \eqref{2.10} and making use too of \eqref{2.19} gives
\begin{equation}\label{2.20}
\rho_{jp|\{1,\cdots,p-1\}}=\frac{A_{jp}}{A_{jj}},\ \ \ \ p+1\leq j\leq N.
\end{equation}
Using now a hyperspherical parametrisation of the lower triangular matrix $A$ by setting $A_{jk}=r_{j}l_{jk}, (r_{j}>0, \, \sum_{j=1}^k A_{jk}^2 = r_j^2)$, with $l_{jk}$ given by \eqref{2.3}, we immediately obtain from \eqref{2.20} the parametrisation \eqref{2.8}.

\setcounter{equation}{0}
\section{The Jacobian, hyperspherical parametrisation of determinant and some consequences}
Let $|J_{\{\rho_{jk}\}\mapsto\{\theta_{jk}\}}|$ denote the Jacobian (absolute value of the determinant of the Jacobian matrix) for the change of variables $\{\rho_{jk}\}\mapsto\{\theta_{jk}\}$ as implied by \eqref{2}, \eqref{4} and \eqref{2.3}. It is shown in \cite{PW15} and \cite{EHNS16} that upon ordering the entries $\{\rho_{21},\rho_{31},\rho_{32},\cdots \}$ i.e.~reading sequentially along rows of the strictly lower triangular portion of $R$, and similarly ordering the angles, the Jacobian matrix is lower triangular. Its determinant and thus the Jacobian can be read off as equal to
\begin{equation}\label{3.1}
|J_{\{\rho_{jk}\}\mapsto\{\theta_{jk}\}}|=\prod_{k=1}^{N-1}\left(\prod_{j=k+1}^N \sin\theta_{jk}\right)^{N-k}.
\end{equation}
Note that this is strictly positive for the angles in the range \eqref{2.4}, and vanishes on the boundary of the range. Earlier, it was shown by Joe \cite{Jo06} (upon adjusting for the different convention by way of projected variables, the details of which were subsequently carried out in
\cite{Ku13}) that in terms of the partial correlations as appearing in \eqref{2.7}
\begin{equation}\label{3.2}
|J_{\{\rho_{jk}\}\mapsto\{\rho_{jk|\{1,\cdots,k-1\}}\}}|=\prod_{j=2}^{N}\prod_{k=1}^{j-1} \left(1-\rho^2_{jk|\{1,\cdots,k-1\}}\right)^{(N-k-1)/2}.
\end{equation}
The expressions \eqref{3.1} and \eqref{3.2} are seem to be consistent with \eqref{2.8}, upon noting that the latter implies
\begin{equation*}
\left|\frac{\partial\rho_{jk|\{1,\cdots,k-1\}}}{\partial\theta_{jk}}\right|=\sin\theta_{jk}.
\end{equation*}

We see from \eqref{3.1} that $|J_{\{\rho_{jk}\}\mapsto\{\theta_{jk}\}}|$ factorises with respect to the variables $\theta_{jk}$. The same is true for $\det R$. Thus, as follows from \eqref{4} and \eqref{2.3} we have \cite{PW15}
\begin{equation}\label{3.3}
\det R=(\det L)^2=\prod_{j=2}^{N}\prod_{p=1}^{j-1}\sin^2\theta_{jp}.
\end{equation}
Joe \cite{Jo06} had earlier shown that in terms of partial correlations
\begin{equation}
\det R=\prod_{j=2}^{N}\prod_{k=1}^{j-1}\left(1-\rho^2_{jk|\{1,\cdots,k-1\}}\right),
\end{equation}
as is consistent with \eqref{3.3} and \eqref{2.8}.

It follows from \eqref{3.2} and \eqref{3.3} that the choice of probability density function on the space of correlation matrices
\begin{equation}\label{3.4}
P(R)=\frac{1}{C_{a,N}}(\det R)^a,\ \ \ \ (a>-1),
\end{equation}
where $C_{a,N}$ denotes the normalisation, permits the hyperspherical parametrisation of the corresponding measure
\begin{equation}\label{3.5}
P(R)(\mathrm d R)=\frac{1}{C_{a,N}}\prod_{j=1}^{N-1}\left(\prod_{k=1}^j\left(\sin\theta_{N+1-k,N-j}\right)^{2a+j}\mathrm d\theta_{j+1,k}\right).
\end{equation}
The explicit formula \eqref{3.5} first appeared in \cite{PW15}, however an equivalent formula in terms of partial correlations can be found in \cite{Jo06}
(again, upon adjusting for the  the different convention by way of projected variables \cite{Ku13}).

A consequence of \eqref{3.5} is the evaluation of the normalisation
\begin{equation}\label{3.6}
C_{a,N}=\prod_{j=1}^{N-1}\left(\int_0^\pi\sin^{2a+j}\theta\mathrm d\theta\right)^j=\prod_{j=1}^{N-1}\left(B\left(a+\frac{j+1}{2},\frac{1}{2}\right)\right)^j,
\end{equation}
where $B(\alpha,\beta)$ is given by \eqref{5.1}. This is already known from \cite{Jo06,PW15}. In fact the probability density \eqref{3.4} on the space of correlation matrices, with $a=(n-N-1)/2$, ($n\leq N, n\in\Z^{+}$), and the evaluation of the normalisation \eqref{3.6} albeit written in a different form, first appeared in the work of Muirhead \cite[Th. 5.1.3]{Mu82}. It is shown there that it can be realised by choosing in \eqref{3} the matrix to be of size $n\times N$ with independent standard Gaussian entries.

Setting $a=0$, it follows that in the case of a uniform distribution
\begin{equation}\label{3.7}
\mathrm{vol}(\mathcal R_N)=\prod_{j=1}^{N-1}\left(B\left(\frac{j+1}{2},\frac{1}{2}\right)\right)^j.
\end{equation}
The working needed to show the equality between this form, which was first given in \cite{PW15}, and the form \eqref{5} as given in \cite{Jo06}, can be found in \cite{PW15}.

We read off from \eqref{3.5} that the marginal distribution of $\theta_{j,1}$, ($j\geq 2$) is proportional to $(\sin\theta_{j1})^{2a+N-1}$. Since for $k=1$, $\rho_{jk|\{1,\cdots,k-1 \}}=\rho_{j1}$, we can then make use of \eqref{2.8} to deduce the marginal distribution of any one $\rho_{jk}$, $(j>k)$ when $R$ has distribution \eqref{3.4}.
\begin{proposition}
In the above setting we have that the marginal distribution of a single non-diagonal element of $R$ has probability density function
\begin{equation}\label{3.8}
\frac{1}{B\left(2a+N-1,\frac{1}{2}\right)}(1-\rho^2)^{2(a-1)+N},\ \ \ \ |\rho|<1.
\end{equation}
\end{proposition}
\begin{remark}
With $a=(n-N-1)/2$, $(n\geq N, n\in\Z^+)$ this result can be found in \cite[\S 5.1 eq.(5)]{Mu82}
\end{remark}
\begin{proposition}
Let $R$ have distribution \eqref{3.4}. We have
\begin{equation}\label{3.9}
\mathbb E(\det R)^s=\prod_{j=1}^{N-1}\left(\frac{B(a+s+(j+1)/2),1/2)}{B(a+(j+1)/2),1/2)}\right)^j
\end{equation}
and
\begin{equation}\label{3.10}
\mathbb E(\log\det R)=\sum_{j=1}^{N-1}j\left(\Psi\left(a+\frac{j+1}{2}\right)-\Psi\left(a+1+\frac{j}{2}\right)\right)
\end{equation}
where $\Psi(x)$ denotes the digamma function.
\end{proposition}
\begin{proof}
The formula \eqref{3.10} is deduced from \eqref{3.9} by differentiating with respect to $s$, and setting $s=0$. The formula \eqref{3.9} is immediate from \eqref{3.3}, \eqref{3.5} and the trigonometric form of the beta function
\begin{equation}
2\int_{0}^{\pi/2}\sin^{2a}\theta\cos^{2b}\theta\mathrm d\theta=B(a,b).
\end{equation}
\end{proof}
\begin{remark}
With $a=(n-N-1)/2$, $(n\geq N, n\in\Z^+)$ a result equivalent to \eqref{3.9} can be found in \cite[\S 5.1 eq.(9)]{Mu82}. It is note in this reference that the result \eqref{3.9} implies that the limiting $a\rightarrow\infty$ distribution of $-2a\log\det R$ is equal to $\chi^2_{N(N-1)/2}$.
\end{remark}

\section{Random correlation matrices with complex or quaternion entries}

A correlation matrix can be constructed out of a data matrix $Y$ with complex entries by replacing $Y^\top Y$ in \eqref{3} by $Y^\top Y$. As mentioned in settings in wireless communications engineering . Of less practical interest, but still of theoretical relevance within random matrix theory (see e.g. \cite{Fo10}) is to form correlation matrices out of data matrices with entries having $2\times 2$ block structure
\begin{equation}\label{4.1}
\begin{bmatrix}
z&w\\-\overline{w}&\overline{z}
\end{bmatrix}.
\end{equation}
Such $2\times 2$ matrices form a representation of quaternions, allowing matrices with quaternion entries to be written as certain structured complex matrices of even size.

The theory relating to the hyperspherical parametrisation of the Cholesky factorisation and its implication for distributions on the space of correlation matrices, as presented in the previous two sections, can readily be extended to the complex and quaternion cases. To begin, augment the notation for the elements in \eqref{2.1} by writing $l_{jk}=l_{jk}^{(F)}$, where $(F)=r,c,q$ for the case of real, complex, quaternion entries respectively. In the complex case, one possible choice, which in fact occurs in the parametrisation of unitary matrices using Euler angles (see e.g. \cite{DF17}), is
\begin{equation*}
l_{jk}^{(c)}=\begin{cases}
l_{jj}^{r}&(j=k)\\
e^{i\psi_{jk}}l_{jk}^{(r)}&(2\leq k\leq j-1),
\end{cases}
\end{equation*}
where $0<\psi_{jk}<2\pi$. However in this parametrisation the Jacobian matrix for the change of variables from $\{\mathrm{Re}(l_{jk}^{(c)}),\mathrm{Im}(l_{jk}^{(c)})\}$ to $\{\psi_{jk},\theta_{jk}\}_{k=1,\cdots,j}$ is not triangular, making the calculations more difficult than need be.

Instead, we write $l_{jk}^{(c)}=l_{jk}^{(c),r}+il_{jk}^{(c),i}$ with $\mathrm{Re}(l_{jk}^{(c)})=l_{jk}^{(c),r}$, $\mathrm{Im}(l_{jk}^{(c)})=l_{jk}^{(c),i}$. Keeping in mind that $l_{jj}^{(c)}$ is required to be real, the analogue of \eqref{2.2a} reads
\begin{equation*}
\left(l_{jj}^{(c),r}\right)^2+\sum_{k=1}^{j-1}\left(\left(l_{jk}^{(c),r}\right)^2+\left(l_{jk}^{(c),r}\right)^2\right)=1.
\end{equation*}
This suggests we view $(l_{j1}^{(c),r},l_{j1}^{(c),i},\cdots,l_{j,j-1}^{(c),r},l_{j,j-1}^{(c),i},l_{j,j}^{(c),r})$ as a point on the real sphere $S_{2j-1}$ and so introduce  the parametrisation
\begin{equation}\label{4.2}
l_{jk}^{(c)}=\begin{cases}
(\cos\theta_{j,2k-1}+i\cos\theta_{j,2k}\sin\theta_{j,2k-1})\prod_{p=1}^{2k-2}\sin\theta_{jp},&(1\leq k\leq j-1)\\
\prod_{p=1}^{2j-2}\sin\theta_{jp},&(j=k).\\
\end{cases}
\end{equation}

In the quaternion case the $j$-th row of $L$ in \eqref{4} consists of $2\times 2$ blocks of the form \eqref{4.1}, each block to be denoted $l_{jk}^{(q)}$, ($k=1,\cdots,j$). The block $l_{jj}^{(q)}$ must represent a real number and this it is required $w=0$ and $z$ be real. For this block we set $z=l_{jj}^{(q),1}$. For the other blocks there are four real numbers corresponding to the real and imaginary parts of $z$ and $w$, which we denote $l_{jk}^{(q),s}$, ($s=1,\cdots,4$). In this setting the analogue of \eqref{2.2a} reads
\begin{equation*}
\left(l_{jj}^{(q),r}\right)^2+\sum_{s=1}^4\sum_{k=1}^{j-1}\left(l_{jk}^{(q),r}\right)^2=1.
\end{equation*}
suggesting that we view
\begin{equation}
(l_{j1}^{(q),1},l_{j1}^{(q),2},l_{j1}^{(q),3},l_{j1}^{(q),4},\cdots,l_{j,j-1}^{(q),1},l_{j,j-1}^{(q),2},l_{j,j-1}^{(q),3},l_{j,j-1}^{(q),4},l_{j,j}^{(q),1})
\end{equation}
as a point on the real sphere $S_{4j-3}$. The corresponding hyperspherical parametrisation is
\begin{align}\label{4.3}
l_{jk}^{(q),s}&=\cos\theta_{j,4(k-1)+s}\left(\prod_{l=1}^{s-1}\sin\theta_{j,4(k-1)+l}\right)\prod_{p=1}^{4(k-1)}\sin\theta_{jp} \: \: (2\leq k\leq j-1; 1 \le s \le 4)\nonumber\\
l_{jj}^{(q),1}&=\prod_{p=1}^{4(j-1)}\sin\theta_{jp}
\end{align}
with the (usual) convention that the products equal unity if they are empty.

Being effectively hyperspherical parametrisations of real spheres, we can write down the corresponding transformation in the volume forms associated with row $j$ in the coordinates \eqref{4.2} and \eqref{4.3}. Thus these will involve the usual Jacobian in such a setting (see e.g. \cite[Th.2.1.3]{Mu82}). In the complex case
\begin{equation}\label{4.4}
\left(\prod_{k=1}^{j-1}\mathrm dl_{jk}^{(c),r}\mathrm dl_{jk}^{(c),i}\right)\mathrm dl_{jj}^{(c),r}=\prod_{p=1}^{2j-2}(\sin\theta_{jp})^{2j-1-p}\mathrm d\theta_{jp},
\end{equation}
while in the quaternion case
\begin{equation}\label{4.5}
\left(\prod_{s=1}^{4}\prod_{k=1}^{j-1}\mathrm dl_{jk}^{(q),s}\right)\mathrm dl_{jj}^{(q),r}=\prod_{p=1}^{4j-4}(\sin\theta_{jp})^{4j-3-p}\mathrm d\theta_{jp}.
\end{equation}
On the LHS of both \eqref{4.4} and \eqref{4.5} it is implicit that the volume form is restricted to the surface of the hypersphere.

We can make use of \eqref{4.2}–\eqref{4.5}  to deduce the analogue of \eqref{3.1} in the complex and quaternion cases.

\begin{proposition}
In the complex case
\begin{equation}\label{4.6}
|J_{\{\rho_{jk}\}\mapsto\{\theta_{jk}\}}^{(c)}|=\prod_{j=2}^{N}\prod_{p=1}^{2j-2}\left(\sin\theta_{jp}\right)^{2N-p-1}
\end{equation}
while in the quaternion case
\begin{equation}\label{4.7}
|J_{\{\rho_{jk}\}\mapsto\{\theta_{jk}\}}^{(q)}|=\prod_{j=2}^{N}\prod_{p=1}^{4j-4}\left(\sin\theta_{jp}\right)^{4N-p-3}
\end{equation}
\end{proposition}
\begin{proof}
Without imposing the constraint that each diagonal entry in $R$ in \eqref{4} equals unity, but still requiring the diagonal entries of $L$ therein to be positive, the change of variables for the volume forms is specified by \cite{DG11}
\begin{equation}\label{4.8}
(\mathrm dR)=2^N\prod_{j=1}^N\left(l_{jj}^{(F),1} \right)^{\beta(N-j)+1}(\mathrm d L),
\end{equation}
where $\beta=1,2,4$ for $(F)=r,c,q$ and $l_{jj}^{(r),1} = l_{jj}^{(r)}$, $l_{jj}^{(c),1} = l_{jj}^{(c),r}$.

With $\rho_{jj}$ denoting the real diagonal entries of $R$ in \eqref{4}, we have that
\begin{equation}
\prod_{j=1}^N\delta(\rho_{jj}-1)=\prod_{j=1}^N\delta\left(\left(l_{jj}^{(F),1}\right)^2-\left(1-\sum_{s=1}^\beta\left(l_{jk}^{(F),s}\right)^2\right) \right).
\end{equation}
These distributions substituted in \eqref{4.8}, upon integrating over $\{\rho_{jj}\}$ on the LHS of \eqref{4.8}, and over $\{l_{jj}^{(F),1}\}$ on the RHS impose the constraints that $\rho_{jj}=1$, $(j=1,\cdots, N)$.

Performing the integrations gives
\begin{equation}\label{4.9}
(\mathrm dR)|_{\rho_{ij}=1}=\prod_{l=1}^N\left(l_{jj}^{(F),1}\right)^{\beta(N-j+1)}(\mathrm dL)|_{\{\ast\}}
\end{equation}
where on the RHS $\ast$ refers to the requirement that
\begin{equation}
l_{jj}^{(F),1}=\left(1-\sum_{s=1}^{\beta}\sum_{k=1}^{j-1}\left(l_{jk}^{(F,s)}\right)^2\right)^{1/2}.
\end{equation}

This latter constraint is built into the hyperspherical parametrisation. Making use then of the $j=k$ case in \eqref{4.2} and \eqref{4.3}, and changing variables in $(\mathrm dL)|_{\{\ast\}}$ by forming the product over $j=2,\cdots,N$ of \eqref{4.4} (complex case) and of \eqref{4.5} (quaternion case), we read off from the resulting forms of the RHS of \eqref{4.9} the stated Jacobians.
\end{proof}

Suppose we now impose on the space of correlation matrices with complex or quaternion entries the probability distribution with density function
\begin{equation}\label{4.10}
\frac{1}{C^{(F)}_{a,N}}\left(\det R\right)^a\ \ \ \ (a>-1),
\end{equation}
in keeping with \eqref{3.4} in the real case. As is conventional in random matrix theory, in the quaternion case $\det R$ is defined as the square root of its value with $R$ represented in terms of the complex blocks \eqref{4.1}. Hence, for $(F)=r,c,q$, we have $\det R=|\det L|^2=\prod_{l=1}^N\left(l_{jj}^{(F)} \right)^2.$ In terms of the hyperspherical parametrisation, reading from \eqref{2.3},\eqref{4.2} and \eqref{4.3} we thus have
\begin{equation}\label{4.11}
\det R=\prod_{j=2}^{ N}\prod_{p=1}^{\beta(j-1)}\sin^2\theta_{j,p},
\end{equation}
where the meaning of $\beta$ is as in \eqref{4.8}. Combining \eqref{4.11} with \eqref{4.4} and \eqref{4.5} allows the normalisation $C^{(F)}$ to be evaluated, as for the derivation of \eqref{3.6}.

\begin{proposition}
With $\beta$ as in \eqref{4.8}, the normalisation \eqref{4.10} has the explicit form
\begin{align}\label{4.12}
C_{a,N}^{(F)}&=\prod_{k=1}^{N-1}\prod_{s=0}^{\beta-1}\left(B\left(a+\frac{\beta k+1-s}{2}, {1 \over 2}\right)\right)^k\nonumber\\
&=\prod_{k=1}^{N-1}\left(\frac{\pi^{\beta/2}\Gamma\left(a+\frac{\beta(k-1)}{2}+1\right)}{\Gamma\left(a+\frac{\beta k}{2}+1\right)}\right)^k\nonumber\nonumber\\
&=\frac{\pi^{\beta(N-1)N/4}}{\left(\Gamma(a+\frac{\beta}{2}(N-1)+1)\right)^{N-1}}\prod_{k=1}^{N-1}\Gamma\left(a+\frac{\beta}{2}(k-1)+1\right)
\end{align}
\end{proposition}
\begin{remark}
	(1) Let $\begin{bmatrix}
r_{jk}^{(F)}
	\end{bmatrix}_{j,k=1}^N$ be a Hermitian matrix with real, complex or quaternion entries for $F=r,c,q$ respectively. Suppose furthermore that all diagonal entries are equal to unity. As emphasized in \cite{EHNS16} in the real case, $1/C^{F}_{0,N}$ is equal to the probability that when the real and imaginary parts (the latter for $F\ne r$) of the strictly upper triangular entries are chosen uniformly at random from $(-1,1)$, the matrix is positive definite and thus a correlation matrix.
	
	(2) It has been commented in the paragraph containing \eqref{3.6} that the probability distribution \eqref{4.10} in the real case can be realised for $a=(n-N-1)/2$ by choosing $Y$ in \eqref{3} to be a standard Gaussian matrix of size $n\times N$. In \cite[Exercises 3.3 q.3]{Fo10} this realisation is extended to all three cases with $a=(\beta/2)(n-N+1-2/\beta)$, $\beta$ as in \eqref{4.8}, and $Y$ a $n\times N$ standard Gaussian matrix with entries from $F$.
\end{remark}

\medskip
Combining \eqref{4.11} with \eqref{4.4} and \eqref{4.5} also allows the complex and quaternion analogues of (\ref{3.8}) and (\ref{3.9}) to
be obtained.

\begin{proposition}
Let $\beta = 1,2,4$ for $F = r,c, q$ respectively, and consider the situation that the random correlation matrix with elements from $F$ is chosen
according to the probability distribution with density (\ref{4.10}). We have that the marginal distribution of the real part of any single non-diagonal
element of $R$ has probability density function
$$
{1 \over B(2a+\beta(N-1))} (1 - \rho^2)^{2a + \beta (N-1) - 1}, \qquad | \rho| < 1.
$$
We also have
$$
\mathbb E  ( \det R)^s = 
\prod_{k=1}^{N-1}\prod_{l =0}^{\beta-1}\left({ B\left(a+s+\frac{\beta k+1-l}{2}, {1 \over 2}\right) 
\over B\left(a+\frac{\beta k+1-l}{2}, {1 \over 2}\right)}
\right)^k
$$
\end{proposition}

In the real case it was shown that the hyperspherical parametrisation (\ref{2.3}) of the Cholesky factorisation (\ref{4}) implies the
simple parametrisation (\ref{2.8}) of the partial correlations $\{ \rho_{jk| \{ 1,\dots, k-1\}} \}_{1 \le k < j \le N}$.
Partial correlations are also well defined in the complex and quaternion cases through the formula (\ref{2.6}).
Note that in the quaternion case this quantity is a $2 \times 2$ matrix with structure (\ref{4.1}). Defining the complex and quaternion
analogues of (\ref{2.9}) and (\ref{2.11}), we see that the working leading to (\ref{2.20}) again holds true, with
$A_{jk}^{(F)} = r_j l_{jk}^{(F)}$. We can thus write down from (\ref{4.2}) and (\ref{4.3}) the corresponding analogues of (\ref{2.8}).

\begin{proposition}
In the setting specified above, and with $ 1 \le k < j \le N$ we have
$$
\rho_{jk|\{1,\cdots,k-1\}}^{(c)}=\cos\theta_{j,2k-1}  + i \cos \theta_{j,2k} \sin \theta_{j,2k-1}
$$
and
$$
\rho_{jk|\{1,\cdots,k-1\}}^{(q)}= \begin{bmatrix} z_{jk}^{(q)} & w_{jk}^{(q)} \\
- \bar{w}_{jk}^{(q)} & \bar{z}_{jk}^{(q)} \end{bmatrix}
$$
with
\begin{align*}
z_{jk}^{(q)} & = \cos \theta_{j, 4k - 3} + i \cos \theta_{j,4k-2} \sin \theta_{j, 4k - 3} \\
w_{j,k}^{(q)} & = \Big ( \cos \theta_{j, 4k-1} + i \cos \theta_{j,4k} \sin \theta_{j,4k-1} \Big )
 \sin \theta_{j, 4k - 3}  \sin \theta_{j, 4k - 2}.
 \end{align*}
 \end{proposition}

\section*{Acknowledgements}
This work is part of a research program supported by the Australian Research Council (ARC) through the ARC Centre of Excellence for Mathematical and Statistical frontiers (ACEMS). PJF also acknowledges partial support from ARC grant DP170102028, and JZ acknowledges the support of a Melbourne postgraduate award, and an ACEMS top up scholarship.

\providecommand{\bysame}{\leavevmode\hbox to3em{\hrulefill}\thinspace}
\providecommand{\MR}{\relax\ifhmode\unskip\space\fi MR }
\providecommand{\MRhref}[2]{%
	\href{http://www.ams.org/mathscinet-getitem?mr=#1}{#2}
}
\providecommand{\href}[2]{#2}


\begin{thebibliography}{10}
	
	\bibitem{An58}
	T.W.~Anderson, \emph{An introduction to multivariate statistics}, Wiley, New
	York, 1958.
	
	\bibitem{CJA10}
	R.M.~Cooke and H.~Joe and K.~Aas, 
	Chapter 3 ``Vines Arise". In: Kurowicka, Dorota; Joe, Harry, editors. \emph{Dependence Modeling: Vine Copula Handbook}, World Scientific, 2011.  
	
	\bibitem{DIP88}  
	J.-M. ~Delosme,  I.C.F.~Ipsen and C.C.~Paige, \emph{ 
		The  Cholesky  factorization,  Schur  complements,  correlation  coefficients,  angles  between  vectors,  and  the  QR 
		factorization}, 
	Technical report, Yale University, 1988 [available on-line]
	
	\bibitem{DF17}
P.~Diaconis and P.J. Forrester, \emph{Hurwitz and the origin of random matrix
  theory in mathematics}, Random Matrix Th. Appl. \textbf{6} (2017), 1730001.
  
 

\bibitem{DG11}
J.A.~D\'iaz-Garc\'ia and R.~Guti\'errex-J\'aimez,  \emph{On {W}ishart distribution: some extensions}, Linear Alg.
  Applications \textbf{435} (2011), 1296--1310.
	
	\bibitem{EHNS16}
	S.~Eastman, S.~Hollis, K.~Numpacharoen and J.~Schlieper,
	\emph{The volume of the spatial region corresponding to $n \times n$ correlation matrices},
	Amer. Math. Monthly \textbf{123} (2016), 909--918.
	
	\bibitem{Fo10}
P.J.~Forrester, \emph{Log-gases and random matrices}, Princeton University Press,
  Princeton, NJ, 2010.
	
	\bibitem{GvL83}
	G.H.~Golub and C.F.~van Loan,
	\emph{Matrix computations}, The John Hopkins Press, Baltimore, MD. 1983.
	
	\bibitem{Hu12}
	W. H\"urlimann, \emph{Positive semi-definite correlation matrices: recursive algorithmic generation and volume
		measure}, Pure Math. Sci. \textbf{1} (2012), 137--149.
	
	\bibitem{Jo06}
	H.~Joe \emph{Generating random correlation matrices based on partial correlations},  J. Multivariate Anal. \textbf{97} (2006), 2177–2189.
	
	\bibitem{Ku13} 
	D.~Kurowicka \emph{Joint density of correlations in the correlation matrix with chordal sparsity patterns},
	J. Multivariate Anal. \textbf{129} (2014), 160--170.
	
	\bibitem{LKJ09}
	D.~Lewandowski, D.~Kurowicka and H.~Joe, \emph{Generating random correlation matrices based on vines and extended onion method},
	J. Multivariate Anal.
	\textbf{100} (2009), 1989–2001.
	
	\bibitem{Mu82}
	R.J. Muirhead, \emph{Aspects of multivariate statistical theory}, Wiley, New
	York, 1982.
	
	\bibitem{LP95}
	M. Laurent and S.~Poljak, \emph{On the positive definite relaxation of the cut polytope}, Linear Alg. Appl. \textbf{223} (1995), 439--461.
	
	\bibitem{Ou81} D.V.~Ouellette, \emph{Schur complements and statistics}, Lin. Algebra Appl. \textbf{36} (1981), 187--295.
	
	\bibitem{PW15}
	M.~Pourahmadi and X.~Wang,  \emph{Distribution of random correlation matrices: hyperspherical
		parameterization of the Cholesky factor},  Stat.. Prob. Letters
	\textbf{106} (2015), 5--12.
	
	\bibitem{PB96} J.D.~Pinheiro and D.M.~Bates, \emph{Unconstrained 
		parameterizations for variance–covariance matrices}, Stat. Comput. \textbf{6} (1996), 289--366.
	
	\bibitem{RBM07}
	F.~Rapisarda, D.~Brigo and F.~Mercurio, \emph{Parameterizing correlations: a geometric interpretation}, IMA J. Manag. Math. \textbf{18}, (2007), 55--73.
	
	\bibitem{RJ00}
	R.~Rebonato and P.~J\"ackel, \emph{The most general methodology for creating a valid correlation matrix for risk management and option pricing purposes},
	J. Risk \textbf{2} (2000), 17--26. 
	
	\bibitem{TV04}
	A.M. Tulino and S.~Verd\'u, \emph{Random matrix theory and wireless
		communications}, Foundations and {T}rends in {C}ommuncations and
	{I}nformation {T}heory, vol.~1, Now Publisher, (2004), pp.~1--182.
	
	\bibitem{Wi28}
	J.~Wishart, \emph{The generalized product moment distribution in samples from a
		normal multivariate population}, Biometrika \textbf{20A} (1928), 32--43.
	
	\bibitem{WCK03} 
	F.~Wong, C.K.~ Carter and R.~Kohn, \emph{Efficient estimation of covariance selection models}, Biometrika \textbf{90}, 809--830.

\end{thebibliography}
\end{document}